\documentclass[10pt,reqno,final]{article}

\usepackage{amsmath,amsfonts,amssymb,amsthm}
\usepackage{mathrsfs,fancybox,pifont}
\usepackage{graphicx,psfrag, float}
\usepackage{booktabs,multirow}
\usepackage{colortbl}
\usepackage{enumitem}
\usepackage{algorithm}
\usepackage[noend]{algpseudocode}
\usepackage{url}
\usepackage[notcite,notref]{showkeys}
\usepackage{color}
\usepackage{cases}
\usepackage{mathtools}
\usepackage{version}
\usepackage{authblk}
\usepackage{subcaption}

\newtheorem{theorem}{\textbf{Theorem}}[section]

\newtheorem{remark}{\textbf{Remark}}[section]
\newtheorem{example}{\textbf{Example}}[section]

\allowdisplaybreaks

\title{Second-Order Linear Relaxation Schemes for Time-Fractional  Phase-Field Models}

\author{Hui Yu\thanks{School of Mathematics and Physics, University of Science and Technology Beijing, Beijing 100083, China (huiyu@xs.ustb.edu.cn).} \,,
Zhaoyang Wang\thanks{Corresponding author. School of Mathematical Sciences, Laboratory of Mathematics and Complex Systems, MOE, Beijing Normal University, Beijing 100875, China; \ Research Center for Mathematics, Beijing Normal University, Zhuhai, Guangdong 519088, China (zhaoyang584520@163.com).} \,,
Ping Lin\thanks{Corresponding author. Division of Mathematics, University of Dundee, Dundee DD1 4HN,United Kingdom (p.lin@dundee.ac.uk).}  }

\affil{}

\date{}

\begin{document}

\maketitle
\begin{abstract}
 This work uses a linear relaxation method to develop efficient numerical schemes for the time-fractional Allen-Cahn and Cahn-Hilliard equations. The $L1^{+}$-CN formula is used to discretize the fractional derivative, and an auxiliary variable is introduced to approximate the nonlinear term by solving an algebraic equation rather than a differential equation as in the invariant energy quadratization (IEQ) and scalar auxiliary variable (SAV) approaches. The proposed semi-discrete scheme is linear, second-order accurate in time, and the inconsistency between the auxiliary and the original variables does not deteriorate over time. Furthermore, we prove that the scheme is unconditionally energy stable. Numerical results demonstrate the effectiveness of the proposed scheme.


  \medskip
\end{abstract}

\section{Introduction}
In recent years, the phase-field approach has been one of the most popular methodologies for simulating multiphase flows and phase‑separation processes. The phase-field model uses an order parameter $\phi$ to describe
the interface of two-fluids or the phase transition between two phases \cite{allen1979microscopic}, and $\phi$ also be employed to characterize the volume or mass fraction of a substance \cite{guo2015thermodynamically, guo2021diffuse}. It is regarded as a branch of gradient flows that evolve toward minimizing the free energy $E(\phi)$, and has been widely applied to materials science (See e.g. \cite{guo2020modeling}), multiphase flow (See e.g. \cite{guo2015thermodynamically, yang2025two}), and tumor growth modeling (See e.g. \cite{wang2025stability}).

Over the past few years, the time-fractional phase-field model \cite{tang2019energy, liu2018time, liao2021energy} has been extensively studied due to its ability to capture memory effects and long‐time scale behavior in phase separation processes. It replaces the classical integer‐order time derivative with a Caputo fractional derivative of order $0<\alpha<1$. In this work, we consider the time-fractional Allen-Cahn (A-C) and Cahn-Hilliard (C-H) equations:
\begin{equation}\label{the1-1}
 \begin{aligned}
 &\frac{\partial ^\alpha}{\partial t^\alpha} \phi= \mathcal{G} \mu, \\
 &\mu = \frac{\delta E}{\delta \phi}=-\epsilon^2\Delta \phi+F'(\phi),
 \end{aligned}
\end{equation}
where $E=\int_\Omega\left(\frac{\epsilon^2}{2}\left|\nabla \phi\right|^2+F(\phi)\right) d \mathbf{x}$ is the free energy, $\epsilon>0$ is an interface width parameter, $F(\phi)=\frac{1}{4} \left(\phi^2-1 \right)^2$ is the double well potential, and $\mathcal{G}=-1$ (A-C) or $\mathcal{G}=\Delta$ (C-H) is a nonpositive operator. Here, the Caputo fractional derivative of $\phi$ is given by
\begin{equation}\nonumber
 \begin{aligned}
\frac{\partial ^\alpha}{\partial t^\alpha}\phi =\frac{1}{\Gamma(1-\alpha)}\int_0^t\frac{\phi'(s)}{(t-s)^\alpha}\, ds,
    \end{aligned}
\end{equation}
where $\Gamma(\cdot)$ is the gamma function.

Similar to the integer-order ($\alpha=1$) phase field equations, the time-fractional A-C and C-H equations also follow certain energy laws. Tang \cite{tang2019energy} et al. proved that the energy is bounded above by the initial energy $E(t)\leq E(0)$ (for $t\in(0,T)$). Quan et al. \cite{quan2020define} subsequently proved that the time-fractional derivative of energy is nonpositive, i.e., $\partial_t^\alpha E(t)\leq 0$.

To solve fractional-order phase-field models, many numerical methods proposed for integer-order phase-field models can be applied, for example, energy quadratization method (IEQ, \cite{guillen2013linear, yang2016linear, yang2017numerical}) and scalar auxiliary variable (SAV) method \cite{shen2019new}. By combining the recently popular SAV method \cite{shen2019new}, Yu et al. \cite{yu2023exponential} constructed a linear, second-order accurate in time and energy stable numerical scheme. The IEQ and SAV methods require replacing the algebraic expression of the auxiliary variable directly by its time derivative to formulate an ordinary differential equation. However, this replacement of an algebraic by its direct time derivative is considered to have weak instability in the context of the differential-algebraic equation, that is, the original algebraic expression of the auxiliary variable will not be well maintained as the time goes (See e.g. \cite{lin1997sequential}). Such a temporal error increase for both IEQ and SAV methods has been observed in literature, e.g., \cite{jiang2022improving} and \cite{alsafri2023numerical}. Recently, based on an idea originally introduced by \cite{besse2002order} for the Schr\"{o}dinger equation, a linear relaxation method for the integer‑order phase‑field model was independently proposed by \cite{zhang2024linear} and \cite{alsafri2023numerical}, the latter further considering coupling with fluid equations. The key idea is that the auxiliary variables are discretized on a time-staggered grid, resulting in that the algebraic expression of the auxiliary variable in either IEQ or SAV is directly solved rather than replacing it by its time derivative, so as to overcome the weak temporal instability known in the numerical treatment in the differential-algebraic equation context (See e.g. \cite{lin1997sequential}).

In this paper, we use the linear relaxation method combined with the $L1^{+}$-CN formula to construct linear, second-order accurate in time, and unconditionally energy stable numerical schemes for the time-fractional A-C and C-H equations. Furthermore, we illustrate through numerical results that the algebraic equation associated with the expression of the auxiliary variable in these schemes does not deteriorate over time, thus ensuring the numerical reliability in long-term computations.

The rest of this paper is organized as follows. In Section \ref{section2}, we develop the $L1^{+}$ linear relaxation schemes and prove their energy stability for the time-fractional Allen–Cahn and Cahn–Hilliard equations. In Section \ref{section3}, we carry out numerical experiments to demonstrate the effectiveness of the proposed schemes. Some conclusions and remarks are given in Section \ref{section4}.

\section{Linear relaxation schemes}
\label{section2}
In this section, we introduce the $L1^{+}$ linear relaxation schemes for the time-fractional C-H and A-C equations and prove their energy stability.

\subsection{The $L1^+$ linear relaxation scheme for the time-fractional C-H equation}

We consider the time-fractional C-H equation:
\begin{equation}\label{pf}
\begin{aligned}
&\partial_t^\alpha \phi= M\Delta \mu,\\
&\mu = -\varepsilon^2  \Delta \phi + f(\phi),
\end{aligned}
\end{equation}
with the periodic or homogeneous Neumann boundary conditions. Here $f(\phi)=F'(\phi)=\left(\phi^2-1\right)\phi$ and $M$ is a constant mobility coefficient.  Based on the idea of linear relaxation method \cite{besse2002order, zhang2024linear, alsafri2023numerical,jiang2023linear}, we introduce a auxiliary variable $r(\phi)$
\begin{equation}
    r(\phi) = \phi^2-1-S,
\end{equation}
where $S\geq 0$ is a stabilization parameter. Thus equation (\ref{pf}) can be reformulated as:
\begin{subequations}\label{the4}
 \begin{align}
 &\partial_t^\alpha \phi = M\Delta \mu, \label{the4a}\\
 &\mu = -\varepsilon^2 \Delta \phi + r\cdot \phi+S\phi, \label{the4b}\\
 &r  = \phi^2-1-S. \label{the4c}
    \end{align}
\end{subequations}

The corresponding modified energy is reformulated in the form of
\begin{equation}\label{modienergy}
    \tilde {E} (\phi,r) := \int_\Omega \left(\frac{ \varepsilon^2}{2}|\nabla \phi|^2 + \frac{1}{2} (r+S)\left(\phi^2-1 -S \right) -\frac{1}{4}r^2\right) \ d \mathbf{x} + \frac{S^2}{4}|\Omega|.
\end{equation}
By substituting Eq. (\ref{the4c}) into Eq.~(\ref{modienergy}), the modified energy $\tilde{E}(\phi, r)$ is demonstrated to be consistent with the original energy $E(\phi)$. Consequently, for a given initial condition $\phi_0$, the relation $\tilde{E}(t_0) = E(t_0)$ holds. Furthermore, the time derivative of the modified energy satisfies
\begin{equation}\label{modified_energy_dt}
    \begin{aligned}
        \frac{d \tilde E(\phi,r)}{dt}  &= \int_\Omega\frac{\delta \tilde E}{\delta \phi}\phi_t + \frac{\delta \tilde E}{\delta r}r'\phi_t d \mathbf{x}\ =\int_\Omega \left(\varepsilon^2 \Delta\phi +\frac{1}{2} (r+S)\cdot2\phi+\frac{1}{2}r^{'}\cdot \left(\phi^2-1-S \right)-\frac{1}{2}rr' \right) \phi_t d\mathbf{x} \\
        &= M^{-1}\int_\Omega\left( \mathcal{G} ^{-1} \partial_t^\alpha\phi\right)\phi_t d \mathbf{x}.
    \end{aligned}
\end{equation}
 Based on the positive definiteness of the kernel function \cite[Corollary 2.1]{tang2019energy}, the modified energy is proven to satisfy the fractional energy dissipation law that
\begin{equation}\label{the7}
    \begin{aligned}
            &\tilde E(\phi(T), r(T))-\tilde E(\phi(0), r(0))  =M^{-1}\int_0^T \int_\Omega\left( \mathcal{G} ^{-1} \partial_t^\alpha\phi\right)\phi_t d \mathbf{x} \, dt\leq 0.
    \end{aligned}
\end{equation}

To discretize the Caputo fractional derivative, we employ the $L1^{+}$ scheme with a non-uniform temporal grid defined by $0 = t_0 < t_1 < \cdots<t_k<\cdots < t_N = T$, where the time step sizes are $\tau_k = t_k - t_{k-1}$ for $k = 1, 2, \ldots, N$. Let $u^k$ denote the approximation of $u$ at time $t_k$, with the discrete difference $\nabla_\tau u^k = u^k - u^{k-1}$ and the midpoint approximation $u^{k-\frac{1}{2}} = \frac{u^k + u^{k-1}}{2}$. For $n \geq 1$, the $L1^{+}$ scheme for the fractional Caputo term is expressed as \cite{ji2020adaptive}
\begin{equation}\label{L1+}
\left( \partial_{\tau}^{\alpha} u \right)^{n-\frac{1}{2}} := \frac{1}{\Gamma(1-\alpha) \tau_n} \int_{t_{n-1}}^{t_n} \sum_{k=1}^n \int_{t_{k-1}}^{\min\{t_k, t\}} (t - s)^{-\alpha} \frac{u^k - u^{k-1}}{\tau_k} \, ds \, dt = \sum_{k=1}^n b_{n-k}^{(n)} \nabla_\tau u^k,
\end{equation}
where the convolution kernel is given by
\begin{equation}
b_{n-k}^{(n)} = \frac{1}{\Gamma(1-\alpha) \tau_n \tau_k} \int_{t_{n-1}}^{t_n} \int_{t_{k-1}}^{\min\{t_k, t\}} (t - s)^{-\alpha} \, ds \, dt.
\end{equation}

By using the relaxation method on the nonlinear term and integrating Eq.(\ref{the4}) over $[t_{n} , t_{n+1}]$, $n\geq 0$, yields
\begin{subequations}\label{L1+chrelaxation}
    \begin{align}
        \frac{1}{\tau_{n+1}}\int_{t_{n}}^{t_{n+1}}\partial_t^\alpha \phi \, dt &=  \frac{1}{\tau_{n+1}}\int_{t_{n }}^{t_{n+1}} M\Delta \mu\, dt \label{L1+chrelaxation1}\\
         \frac{1}{\tau_{n+1}}\int_{t_{n}}^{t_{n+1}}\mu \,dt&=  \frac{1}{\tau_{n+1}}\int_{t_{n}}^{t_{n+1}}\left(-\varepsilon^2 \Delta \phi + \left(r +S \right)\phi\right)\,dt\\
         \frac{1}{\tau_{n+1}}\int_{t_n}^{t_{n+1}}r \,dt & =  \frac{1}{\tau_{n+1}}\int_{t_{n}}^{t_{n+1}}\left(\phi^2-1 -S\right)\, dt.
    \end{align}
\end{subequations}
Adapting the $L1^{+}$ formulation (\ref{L1+}) to Eq.(\ref{L1+chrelaxation1}), the $L1^{+}$ linear relaxation scheme for the system (\ref{the4}) is derived as
\begin{subequations}\label{Dis_RRER_CH}
    \begin{align}
        &\partial_\tau^\alpha \phi^{n+\frac{1}{2}} =  M\Delta \mu^{n+\frac{1}{2}},\label{Dis_RRER_CH1}\\
        &\mu^{n+\frac{1}{2}} = -\varepsilon^2 \Delta \phi^{n+\frac{1}{2}} +   \left(r^{n+\frac{1}{2}} +S \right)\phi^{n+\frac{1}{2}},\label{Dis_RRER_CH2}\\
        &\frac{r^{n+\frac{1}{2}} +r^{n-\frac{1}{2}} }{2} = \left(\phi^n\right)^2-1 -S,\label{Dis_RRER_CH3}
    \end{align}
\end{subequations}
where $\phi^{n+\frac{1}{2}}$ and $\mu^{n+\frac{1}{2}}$ denote the approximations at the midpoint $t_{n+\frac{1}{2}}$.
For $n=0$, the initial condition is defined such that $r^{\frac{1}{2}} =   \phi^0(1-\phi^0) - S$, which implies $r^{-\frac{1}{2}} = \phi^0(1-\phi^0) - S$. Consequently, for a given initial value $\phi^0$, the modified energy satisfies $\tilde{E}\left( \phi^0, r^{-\frac{1}{2}} \right) = E\left( \phi^0 \right)$.

Furthermore, we prove that scheme (\ref{Dis_RRER_CH}) preserves the energy stability corresponding to (\ref{the7}).

\begin{theorem}\label{theorem_energy_TFCH}
The $L1^+$ linear relaxation scheme (\ref{Dis_RRER_CH}) preserves the modified energy dissipation law
    \begin{equation}\label{disenergylawch}
        \tilde E\left( \phi^{n+1} , r^{n+\frac{1}{2}}\right)\leq   \tilde E\left( \phi^{0} , r^{-\frac{1}{2}}\right)=  E\left( \phi^0\right),
    \end{equation}
where
\begin{equation}
    \tilde {E} (\phi^{n+1},r^{n+\frac{1}{2}}) = \int_\Omega \left(\frac{\varepsilon^2}{2}|\nabla \phi^{n+1}|^2 + \frac{1}{2} \left(r^{n+\frac{1}{2}} +S\right)\left( \left(\phi^{n+1}\right)^2 -1  - S \right) -\frac{1}{4}\left( r^{n+\frac{1}{2}} \right)^2\right) \ d \mathbf{x} + \frac{S^2}{4}|\Omega|.
\end{equation}
\end{theorem}

\begin{proof}
    Taking the inner product of Eq.~\eqref{Dis_RRER_CH1} with $  M^{-1}\Delta^{-1}\nabla_\tau \phi^{k+1}$ and Eq.~\eqref{Dis_RRER_CH2} with $\phi^{k+1}-\phi^k$ in $L^2$ space, and summing the resulting equations, we derive that
    \begin{equation} \label{proof1}
         M^{-1}\left( \Delta^{-1}\nabla_\tau\phi^{k+1}, \partial_\tau^\alpha \phi^{k+ \frac{1}{2}}\right)= \left( \Delta \mu^{k+\frac{1}{2}}, \Delta^{-1}\nabla_\tau\phi^{k+1}\right) =\left(  \mu^{k+\frac{1}{2}},  \nabla_\tau\phi^{k+1}\right).
    \end{equation}
    For the left term of Eq.(\ref{proof1}), by setting $\psi^k = \nabla\left(\Delta^{-1} \phi^k\right)$, it holds that
        \begin{equation}\label{caputo_inner}
        \left( \Delta^{-1}\nabla_\tau\phi^{k+1}, \partial_\tau^\alpha \phi^{k+ \frac{1}{2}}\right)= -\left( \nabla_\tau  \psi^{k+1}  ,  \partial_\tau^\alpha \psi^{k+ \frac{1}{2}}\right)  \leq 0.
    \end{equation}
Notice that the following relationship holds
\begin{equation}
    \begin{aligned}
(r^{k+\frac{1}{2}} + S) \frac{\phi^{k} + \phi^{k+1}}{2} \nabla_{\tau} \phi^{k+1}
&= \frac{1}{2} (r^{k+\frac{1}{2}} +S) \left[ (\phi^{k+1})^2 - (\phi^{k})^2 \right]
- \frac{1}{2} (r^{k-\frac{1}{2}} + S) \left[ (\phi^{k})^2 - (\phi^{k})^2 \right] \notag \\
&= \frac{1}{2} (r^{k+\frac{1}{2}} + S) (\phi^{k+1})^2
- \frac{1}{2} (r^{k-\frac{1}{2}} + S) (\phi^{k})^2
- \frac{1}{2} (r^{k+\frac{1}{2}} - r^{k-\frac{1}{2}}) (\phi^{k})^2. \label{eq:19}
\end{aligned}
\end{equation}
Thus, for the right term of Eq.(\ref{proof1}), there is
\begin{equation}\label{proof2}
    \begin{aligned}
(\nabla_{\tau} \phi^{k+1}, u^{k+\frac{1}{2}})
&= \left( \varepsilon^2 \Delta \frac{\phi^{k} + \phi^{k+1}}{2} + (r^{k+\frac{1}{2}} +S) \frac{\phi^{k} + \phi^{k+1}}{2}, \nabla_{\tau} \phi^{k+1} \right)  \\
&= \frac{\varepsilon^2}{2} \left( \|\nabla \phi^{k+1}\|^2_{L^2} - \|\nabla \phi^{k}\|^2_{L^2} \right)
+ \left( (r^{k+\frac{1}{2}} +S) \frac{\phi^{k} + \phi^{k+1}}{2}, \nabla_{\tau} \phi^{k+1} \right) \\
&= \frac{\varepsilon^2}{2} \left( \|\nabla \phi^{k+1}\|^2_{L^2} - \|\nabla \phi^{k}\|^2_{L^2} \right)
+ \frac{1}{2} \left( (r^{k+\frac{1}{2}} +S) \left(  (\phi^{k+1})^2-1-S\right), 1 \right)  \\
&\quad - \frac{1}{2} \left( (r^{k-\frac{1}{2}} +S) \left( (\phi^{k})^2-1 -S\right), 1 \right)
- \frac{1}{4} \left( (r^{k+\frac{1}{2}})^2 - (r^{k-\frac{1}{2}})^2, 1 \right)  \\
&= \tilde{E}\left( \phi^{k+1}, r^{k+\frac{1}{2}}\right) - \tilde{E}\left( \phi^{k }, r^{k-\frac{1}{2}}\right).
\end{aligned}
\end{equation}
Combining Eq.(\ref{caputo_inner}) and Eq.(\ref{proof2}), we demonstrate that
    \begin{equation}\label{Dis_RRER_CH2_ener}
        \tilde E\left(\phi^{k+1},r^{k+\frac{1}{2}}\right) - \tilde E\left(\phi^{k},r^{k-\frac{1}{2}}\right) = -\left( \nabla_\tau\psi^{k+1} , \partial_\tau^\alpha \psi^{k+\frac{1}{2}} \right) \leq 0.
    \end{equation}

By summing $\tilde E\left(\phi^{k+1},r^{k+\frac{1}{2}}\right) - \tilde E\left(\phi^{k},r^{k-\frac{1}{2}}\right)$ from $k=0$ to $n$, the discrete energy boundedness law (\ref{disenergylawch}) is established for the $L1^{+}$ linear relaxation scheme (\ref{Dis_RRER_CH}).

\end{proof}

\begin{remark}
We show below the connection between the modified energy and the original energy. Note that
\begin{equation}
    \begin{aligned}
        E(\phi^{n+1}) &\approx \tilde E\left( \phi^{n+1}, r^{n+1}\right)\\
         &= \int_\Omega \left(\frac{\varepsilon^2}{2}|\nabla \phi^{n+1}|^2 + \frac{1}{2} \left(r^{n+1} +S\right)\left( \left(\phi^{n+1}\right)^2 -1  - S \right) -\frac{1}{4}\left( r^{n+1} \right)^2\right) \ d \mathbf{x} + \frac{S^2}{4}|\Omega|\\
        &= \int_\Omega \left(\frac{\varepsilon^2}{2}|\nabla \phi^{n+1}|^2 + \frac{1}{2} \left(r^{n+\frac{1}{2}} +S\right)\left( \left(\phi^{n+1}\right)^2 -1  - S \right) -\frac{1}{4}\left( r^{n+\frac{1}{2}} \right)^2\right) \ d \mathbf{x} + \frac{S^2}{4}|\Omega|\\
        &\quad +\int_\Omega\frac{1}{2}\left( r^{n+1}-r^{n+\frac{1}{2}} \right)\left( \left(\phi^{n+1}\right)^2 -1  - S \right) -\frac{1}{4}\left( \left( r^{n+1} \right)^2-  \left( r^{n+\frac{1}{2}} \right)^2\right) d\mathbf{x}\\
        & \approx \tilde E\left(\phi^{n+1}, r^{n+\frac{1}{2}}\right)+\frac{1}{4}\left\|r^{n+1}-r^{n+\frac{1}{2}}  \right\|_{L^2}^2.
    \end{aligned}
\end{equation}
The variables are second-order accurate in time. In the discrete case, the modified energy is a second-order approximation of the original energy.
\end{remark}

\subsection{The $L1^+$ linear relaxation scheme for the time-fractional A-C equation}
For the following time-fractional A-C equation
\begin{equation}\label{TFAC}
    \partial_t^\alpha \phi = M\left(\varepsilon^2\Delta \phi + (\phi^2-1)\phi\right)
\end{equation}
with the free energy $E(\phi) =\int_\Omega \left(\frac{\varepsilon^2}{2}|\nabla \phi|^2 + \frac{1}{4}\left(\phi^2-1\right)^2\right) d\mathbf{x}$, we have the following similar second-order linear relaxation scheme:
\begin{subequations}\label{Dis_RRER_AC}
    \begin{align}
        &\partial_\tau^\alpha \phi^{n+\frac{1}{2}} =  M\mu^{n+\frac{1}{2}},\label{Dis_RRER_AC1}\\
        &\mu^{n+\frac{1}{2}} = -\varepsilon^2 \Delta \phi^{n+\frac{1}{2}} +  \left(r^{n+\frac{1}{2}} + S \right) \phi^{n+\frac{1}{2}} ,\label{Dis_RRER_AC2}\\
        &\frac{r^{n+\frac{1}{2}} +r^{n-\frac{1}{2}} }{2} = \left(\phi^n \right)^2-1 - S.\label{Dis_RRER_AC3}
    \end{align}
\end{subequations}

\begin{theorem}\label{theorem_energy_TFAC}
    The $L1^+$ linear relaxation scheme (\ref{Dis_RRER_AC}) preserves the modified energy dissipation law
    \begin{equation}\label{disenergylaw}
        \tilde E\left( \phi^{n+1} , r^{n+\frac{1}{2}}\right)\leq   \tilde E\left( \phi^{0} , r^{-\frac{1}{2}}\right)=  E\left( \phi^0\right),
    \end{equation}
where
\begin{equation}
    \tilde {E} (\phi^{n+1},r^{n+\frac{1}{2}}) = \int_\Omega \frac{\varepsilon^2}{2}|\nabla \phi^{n+1}|^2 + \frac{1}{2} \left(r^{n+\frac{1}{2}} +S\right)\left( (\phi^{n+1})^2 -1  - S \right) -\frac{1}{4}\left( r^{n+\frac{1}{2}} \right)^2 \ d \mathbf{x} + \frac{S^2}{4}|\Omega|.
\end{equation}
\end{theorem}

\begin{proof}
As the argument is similar to that of Theorem \ref{theorem_energy_TFCH}, we omit the proof here for brevity.
\end{proof}

\section{Numerical results}
\label{section3}
This section presents two numerical examples to illustrate the effectiveness of the proposed numerical method. We consider the periodic boundary condition on the domain $[0,\pi]^2$ for the time-fractional A-C and C-H equations, and use the Fourier spectral method for spatial discretization. The time steps are designed as $t_n = \left(\frac{n}{N}\right)^{rNt}$. Without specific needs, some parameters are fixed to be $S=2$, $\epsilon=0.2$, $M=0.1$.

\begin{example}
(Convergence tests)
We first present an example in which appropriate source terms are selected so that the time-fractional A-C and C-H models admit the following exact solution:
\begin{equation}
    \phi(x,y,t) = \left(1-t^{2.5}\right)\left( \frac{1}{4}\sin(2x)\cos(2y) + 0.45\right),\ t\in[0,0.5].
\end{equation}
\end{example}

 \begin{table}[ht]
\centering
\caption{$L^\infty$-errors and convergence rates of $\phi$ and $r$ for the time-fractional A-C equation at $t=0.5$.}
\begin{tabular}{c c c c c c c c c c}
\toprule
 & & \multicolumn{2}{c}{$\alpha = 0.3$} & \multicolumn{2}{c}{$\alpha = 0.6$} & \multicolumn{2}{c}{$\alpha = 0.9$} & \multicolumn{2}{c}{$\alpha = 1$} \\
\cmidrule(lr){3-4} \cmidrule(lr){5-6} \cmidrule(lr){7-8} \cmidrule(lr){9-10}
 & $N $ & Error & Order & Error & Order & Error & Order & Error & Order \\
\midrule
\multirow{4}{*}{$\phi$} & 8   & $2.62E{-3}$ & -- & $3.36E{-4}$ & -- & $5.02E{-5}$ & -- & $1.63E{-5}$ & -- \\
 & 16  & $6.91E{-4}$ & 1.92 & $8.81E{-5}$ & 1.93 & $1.34E{-5}$ & 1.91 & $4.14E{-6}$ & 1.98 \\
 & 32  & $1.77E{-4}$ & 1.97 & $2.27E{-5}$ & 1.95 & $3.55E{-6}$ & 1.92 & $1.05E{-6}$ & 1.98 \\
 & 64  & $4.47E{-5}$ & 1.98 & $5.82E{-6}$ & 1.97 & $9.32E{-7}$ & 1.93 & $2.63E{-7}$ & 1.99 \\
\midrule
\multirow{4}{*}{$r$} & 8   & $8.75E{-3}$ & -- & $1.24E{-3}$ & -- & $2.18E{-4}$ & -- & $7.93E{-5}$ & -- \\
 & 16  & $2.46E{-3}$ & 1.83 & $3.20E{-4}$ & 1.95 & $5.54E{-5}$ & 1.98 & $1.93E{-5}$ & 2.04 \\
 & 32  & $6.48E{-4}$ & 1.92 & $8.16E{-5}$ & 1.97 & $1.41E{-5}$ & 1.98 & $4.75E{-6}$ & 2.03 \\
 & 64  & $1.66E{-4}$ & 1.96 & $2.07E{-5}$ & 1.98 & $3.57E{-6}$ & 1.98 & $1.17E{-6}$ & 2.02 \\
\bottomrule
\end{tabular}
\label{tab:AC}
\end{table}

\begin{table}[ht]
\centering
\caption{$L^\infty$-errors and convergence rates of $\phi$ and $r$ for the time fractional C-H equation at $t=0.5$.}
\begin{tabular}{c c c c c c c c c c}
\toprule
 & & \multicolumn{2}{c}{$\alpha = 0.3$} & \multicolumn{2}{c}{$\alpha = 0.6$} & \multicolumn{2}{c}{$\alpha = 0.9$} & \multicolumn{2}{c}{$\alpha = 1$} \\
\cmidrule(lr){3-4} \cmidrule(lr){5-6} \cmidrule(lr){7-8} \cmidrule(lr){9-10}
 & $N $ & Error & Order & Error & Order & Error & Order & Error & Order \\
\midrule
\multirow{4}{*}{$\phi$} & 8   & $2.22E{-2}$ & -- & $3.61E{-3}$ & -- & $5.99E{-4}$ & -- & $2.36E{-4}$ & -- \\
 & 16  & $7.23E{-3}$ & 1.62 & $1.03E{-3}$ & 1.81 & $1.64E{-4}$ & 1.87 & $6.16E{-5}$ & 1.94 \\
 & 32  & $2.06E{-3}$ & 1.81 & $2.77E{-4}$ & 1.89 & $4.42E{-5}$ & 1.89 & $1.58E{-5}$ & 1.96 \\
 & 64  & $5.43E{-4}$ & 1.92 & $7.25E{-5}$ & 1.93 & $1.18E{-5}$ & 1.91 & $4.02E{-6}$ & 1.97 \\
\midrule
\multirow{4}{*}{$r$} & 8   & $4.42E{-2}$ & -- & $1.51E{-2}$ & -- & $3.26E{-3}$ & -- & $1.44E{-3}$ & -- \\
 & 16  & $2.42E{-2}$ & 0.87 & $4.38E{-3}$ & 1.79 & $8.30E{-4}$ & 1.97 & $3.41E{-4}$ & 2.08 \\
 & 32  & $8.13E{-3}$ & 1.57 & $1.17E{-3}$ & 1.91 & $2.11E{-4}$ & 1.98 & $8.16E{-5}$ & 2.06 \\
 & 64  & $2.32E{-3}$ & 1.81 & $3.02E{-4}$ & 1.95 & $5.35E{-5}$ & 1.98 & $1.97E{-5}$ & 2.05 \\
\bottomrule
\end{tabular}
\label{tab:CH}
\end{table}

In the simulation, we use graded temporal meshes, with $rNt = 8, 3, 1.5, 1$ for $\alpha = 0.3, 0.6, 0.9, 1$, respectively. The $L^\infty$ errors and convergence rates with different $\alpha$ are listed in Tables \ref{tab:AC} and \ref{tab:CH}, which show that the proposed schemes can achieve second-order accuracy for the time-fractional A-C and C-H equations.

\begin{example}
We consider the time-fractional A-C and C-H equations with the following initial condition:
\begin{equation}
    \phi_0 =  \tanh\left(\frac{0.5 - \sqrt{\left({\frac{x}{0.5}}\right)^2+\left({\frac{y}{0.3}}\right)^2}}{ \sqrt{2}\epsilon }\right).
\end{equation}
\end{example}

\begin{figure}[H]
  \centering
  \begin{subfigure}[b]{0.45\textwidth}
    \centering
    \includegraphics[width=\textwidth]{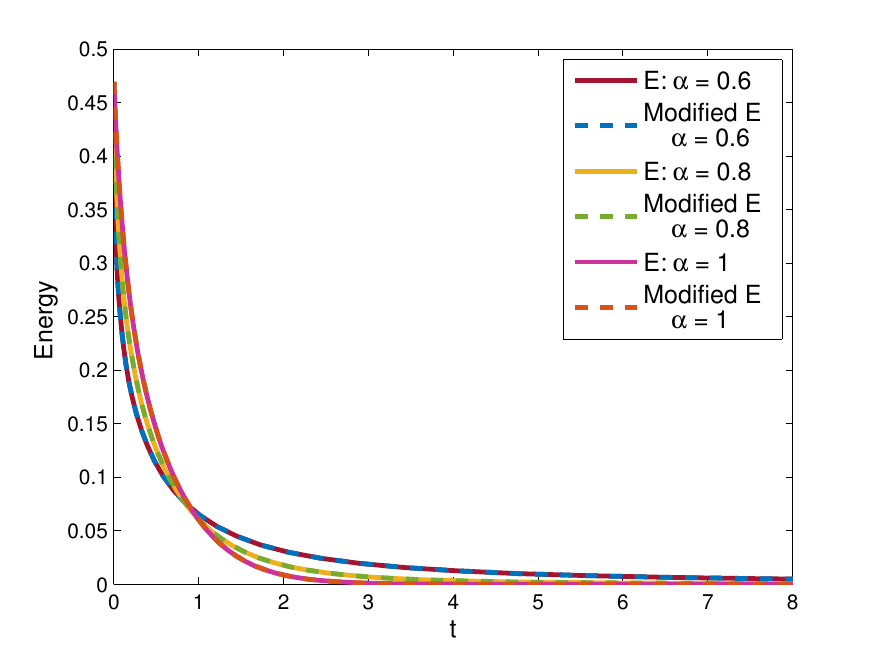}
    \caption{}
    \label{fig:AC_energy}
  \end{subfigure}
  \hfill
    \begin{subfigure}[b]{0.45\textwidth}
      \centering
      \includegraphics[width=\textwidth]{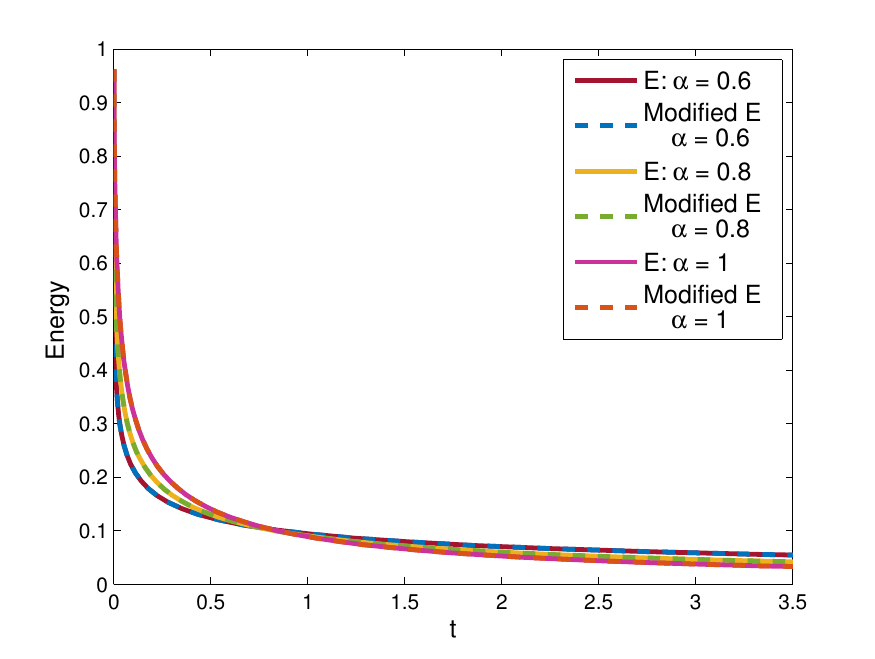}
      \caption{}
      \label{fig:CH_Energy}
    \end{subfigure}
  \caption{Original energy and modified energy with different $\alpha$. (a) The time-fractional A-C equation. (b) The time-fractional C-H equation.
}
  \label{fig:AC}
\end{figure}

\begin{figure}[H]
  \centering
  \begin{subfigure}[b]{0.45\textwidth}
    \centering
    \includegraphics[width=\textwidth]{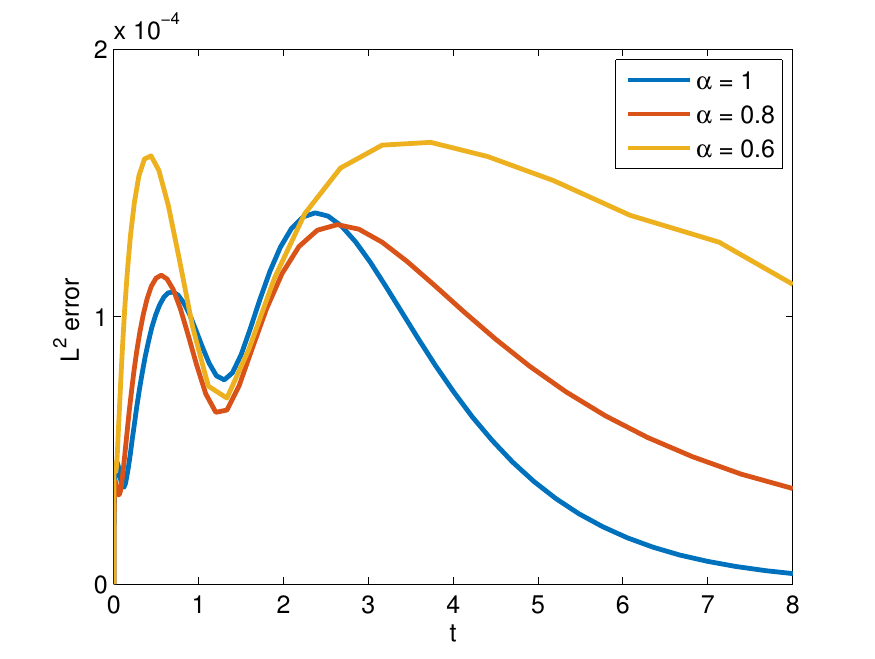}
    \caption{}
    \label{fig:AC_r_L2}
  \end{subfigure}
  \hfill
  \begin{subfigure}[b]{0.45\textwidth}
    \centering
    \includegraphics[width=\textwidth]{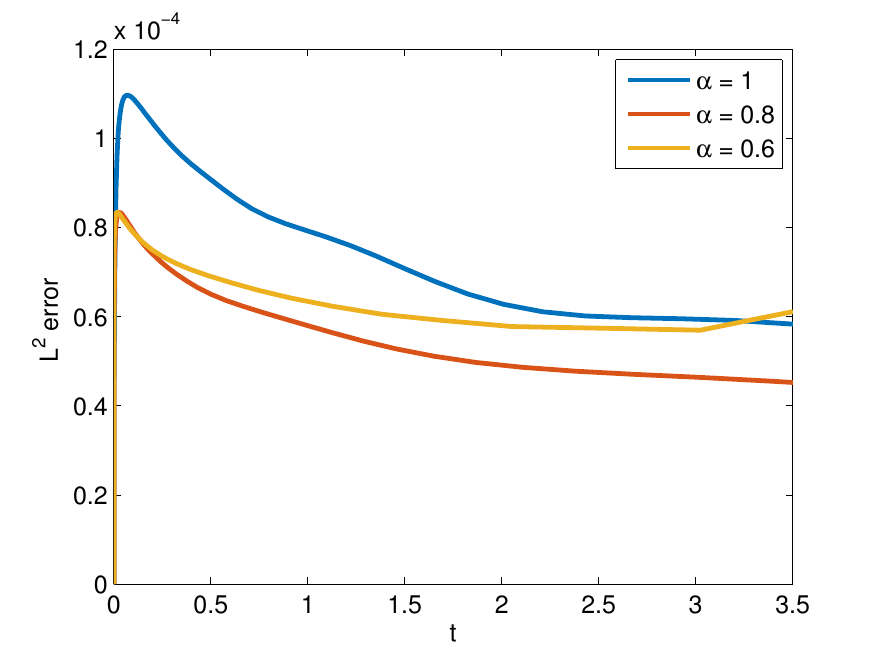}
    \caption{}
    \label{fig:CH_r_L2}
  \end{subfigure}
  \caption{$L^2$-error between auxiliary variable $r^{n+1/2}$ and original variable $(\phi^{n+1/2})^2-1-S$. (a) The time-fractional A-C equation. (b) The time-fractional C-H equation.}
  \label{fig:CH}
\end{figure}

Figures \ref{fig:AC}(a) and (b) show that the variation of the original and modified energies with different $\alpha$ for the time-fractional A-C and C-H equations. We can observe that the modified energy is highly consistent with the original energy until the steady state. Furthermore, we investigate the numerical errors between the auxiliary variable $r$ and the original variable $\phi^2-1-S$. As shown in Figures \ref{fig:CH}(a) and (b), the errors are related to the phase transition rate. The errors increase when the phase changes abruptly, and decrease as it gradually approaches a steady state. However, for the IEQ and SAV methods, since the auxiliary variable $r$ is obtained by solving a differential equation, the numerical errors between the original and auxiliary variables or between the original energy and the modified energy, accumulate over time (see \cite{jiang2022improving}), which is well known in the community. We note that the proposed linear relaxation scheme introduces an algebraic equation to solve the auxiliary variable, so that the numerical inconsistency between the original and auxiliary variables does not deteriorate over time. This provides a new perspective for designing structure-preserving numerical schemes for phase-field models in the future.

\section{Conclusions and remarks}
\label{section4}
In this paper, we develop $L1^{+}$ linear relaxation schemes for solving the time-fractional A-C and C-H equations. An auxiliary variable is introduced to approximate the original nonlinear term. Time-stepping is then carried out numerically by directly solving the algebraic equation associated with the expression of the auxiliary variable, which prevents the error accumulation over time observed in IEQ and SAV methods, where the algebraic equation is replaced by its time derivative. We rigorously prove the unconditional energy stability of the numerical schemes. Numerical results demonstrate the effectiveness of the proposed numerical schemes. In our future work, we shall consider using the ideas in this paper to design efficient structure-preserving schemes for the coupled fluid and phase‑field problems.

\section*{Acknowledgments}
H. Yu is partially supported by the China Scholarship Council for one year of research at the University of Dundee. Z. Wang is partially supported by the China Postdoctoral Science Foundation under grant 2024M760239. P. Lin is partially supported by the National Natural Science Foundation of China 12371388, 11861131004.

\bibliographystyle{elsarticle-num}
\bibliography{Ref}
\end{document}